\documentclass{article}%
\usepackage{amsmath}
\usepackage{amsfonts}
\usepackage{amssymb}
\usepackage{graphicx}%
\providecommand{\U}[1]{\protect \rule{.1in}{.1in}}
\newtheorem{theorem}{Theorem}

\newtheorem{corollary}[theorem]{Corollary}

\newtheorem{proposition}[theorem]{Proposition}
\newtheorem{remark}[theorem]{Remark}

\newenvironment{proof}[1][Proof]{\noindent \textbf{#1.} }{\  \rule{0.5em}{0.5em}}
\begin{document}

\title{\textbf{Covariation representations for Hermitian L\'{e}vy process ensembles
of free infinitely divisible distributions }}
\author{J. Armando Dom\'{\i}nguez-Molina\thanks{jadguez@uas.edu.mx}\\Facultad de Ciencias F\'{\i}sico-Matem\'{a}ticas\\Universidad Aut\'{o}noma de Sinaloa, M\'{e}xico
\and V\'{\i}ctor P\'{e}rez-Abreu\thanks{pabreu@cimat.mx}\\Departamento de Probabilidad y Estad\'{\i}stica\\CIMAT, Guanajuato, M\'{e}xico
\and Alfonso Rocha-Arteaga\thanks{arteaga@uas.edu.mx}\\Facultad de Ciencias F\'{\i}sico-Matem\'{a}ticas\\Universidad Aut\'{o}noma de Sinaloa, M\'{e}xico}
\maketitle

\begin{abstract}
It is known that the so-called Bercovici-Pata bijection can be explained in
terms of certain Hermitian random matrix ensembles $\left(  M_{d}\right)
_{d\geq1}$ whose asymptotic spectral distributions are free infinitely
divisible. We investigate Hermitian L\'{e}vy processes with jumps of rank one
associated to these random matrix ensembles introduced in \cite{BG} and
\cite{CD}. A sample path approximation by covariation processes for these
matrix L\'{e}vy processes is obtained. As a general result we prove that any
$d\times d$ complex matrix subordinator with jumps of rank one is the
quadratic variation of an $\mathbb{C}^{d}$-valued L\'{e}vy process. In
particular, we have the corresponding result for matrix subordinators with
jumps of rank one associated to the random matrix ensembles $\left(
M_{d}\right)  _{d\geq1}$.



\end{abstract}

\section{Introduction}

New models of infinitely divisible random matrices have emerged in recent
years from both applications and theory. On the one hand, they have been
important in multivariate financial L\'{e}vy modelling where stochastic
volatility models have been proposed using L\'{e}vy and Ornstein-Uhlenbeck
matrix valued processes; see \cite{BNSe07}, \cite{BNSe09}, \cite{BNS11} and
\cite{PiSe09a}. A key role in these models is played by the positive-definite
matrix processes and more general matrix covariation processes.

On the other hand, in the context of free probability, Bercovici and Pata
\cite{BP} introduced a bijection $\Lambda$ from the set of classical
infinitely divisible distributions to the set of free infinitely divisible
distributions. This bijection was explained in terms of random matrix
ensembles by Benaych-Georges \cite{BG} and Cabanal-Duvillard \cite{CD},
providing in a more palpable way the bijection $\Lambda$ and producing a new
kind of infinitely divisible random matrix ensembles. Moreover, the results in
\cite{BG} and \cite{CD} constitute a generalization of Wigner's result for the
Gaussian Unitary Ensemble and give an alternative simple infinitely divisible
random matrix model for the Marchenko-Pastur distribution, for which the
Wishart and other empirical covariance matrix ensembles are not infinitely divisible.

More specifically, it is shown in \cite{BG} and \cite{CD} that for any
one-dimensional infinitely divisible distribution $\mu$ there is an ensemble
of Hermitian random matrices $(M_{d})_{d\geq1}$, whose empirical spectral
distribution converges weakly almost surely to $\Lambda(\mu)$ as $d$ goes to
infinity. Moreover, for each $d\geq1$, $M_{d}$ has a unitary invariant matrix
distribution which is also infinitely divisible in the matrix sense. From now
on we call these models BGCD matrix ensembles. We consider additional facts of
BGCD models in Section 3.

A problem of further interest is to understand the matrix L\'{e}vy processes
$\left \{  M_{d}(t)\right \}  _{t\geq0}$ associated to the BGCD matrix
ensembles. It was pointed out in \cite{DRA}, \cite{PAS} that the L\'{e}vy
measures of these models are concentrated on rank one matrices. This means
that the random matrix $M_{d}$ is a realization, at time one, of a matrix
valued L\'{e}vy process $\left \{  M_{d}(t)\right \}  _{t\geq0}$ with rank one
jumps $\Delta M_{d}(t)=M_{d}(t)-M_{d}(t-).$

The purpose of this paper is to study the structure of a $d\times d$ Hermitian
L\'{e}vy process $\left \{  L_{d}(t)\right \}  _{t\geq0}$ with rank one jumps.
It is shown in Section 4 that if $L_{d}$ is a $d\times d$ complex matrix
subordinator, it is the quadratic variation of an $\mathbb{C}^{d}$-valued
L\'{e}vy process $X_{d}$, being the converse and extension of a known result
in dimension one, see \cite[Example 8.5]{CT}. The process $X_{d}$ is
constructed via its L\'{e}vy-It\^{o} decomposition. In\ Section 5 we consider
new realizations in terms of covariation of $\mathbb{C}^{d}$-valued L\'{e}vy
process for matrix compound Poisson process as well as sample path
approximations for L\'{e}vy processes associated to general BGCD ensembles. A
new insight on Marchenko-Pastur's type results for empirical covariance matrix
ensembles was recently given in \cite{BGCD} by considering compound Poisson
models (then infinitely divisible). In this direction our results show the
role of covariation of $d$-dimensional L\'{e}vy processes as an alternative to
empirical covariance processes.

For convenience of the reader, and since the material and notation in the
literature is disperse and incomplete, we include Section 2 with a review on
preliminaries on complex matrix semimartingales and matrix valued L\'{e}vy
processes that are used later on in this paper.

\section{Preliminaries on matrix semimartingales and matrix L\'{e}vy \newline
processes}

Let $\mathbb{M}_{d\times q}=\mathbb{M}_{d\times q}\left(  \mathbb{C}\right)  $
denote the linear space of $d\times q$ matrices with complex (respectively
real) entries with scalar product $\left \langle A,B\right \rangle
=~\mathrm{tr}\left(  AB^{\ast}\right)  $ and the Frobenius norm $\left \Vert
A\right \Vert =\left[  \mathrm{tr}\left(  AA^{\ast}\right)  \right]  ^{1/2}$
where $\mathrm{tr}$ denotes the (non normalized) trace. If $q=d,$ we write
$\mathbb{M}_{d}=\mathbb{M}_{d\times d}$. The set of Hermitian random matrices
in $\mathbb{M}_{d}$ is denoted by $\mathbb{H}_{d}$. Likewise, let
$\mathbb{U}_{d\times q}=\mathbb{U}_{d\times q}\left(  \mathbb{C}\right)
=\left \{  U\in \mathbb{M}_{d\times q}:U^{\ast}U=\mathrm{I}_{q}\right \}  .$ If
$q=d,$ $\mathbb{U}_{d}=\mathbb{U}_{d\times d}$.

We denote by $\mathbb{H}_{d(1)}$ the set of matrices in $\mathbb{H}_{d}$ of
rank one and by $\mathbb{H}_{d}^{+}$ ($\overline{\mathbb{H}}_{d}^{+}$) the set
of positive (nonnegative) definite matrices in $\mathbb{H}_{d}$. Likewise
$\mathbb{H}_{d(1)}^{+}=\mathbb{H}_{d(1)}\cap \overline{\mathbb{H}}_{d}^{+}$ is
the closed cone of $d\times d$ nonnegative definite matrices of rank one. Let
$\mathbb{S}(\mathbb{H}_{d(1)})$ denote the unit sphere of $\mathbb{H}_{d(1)}$.

\begin{remark}
\label{Decomp}(a) Every $V\in \mathbb{H}_{d(1)}^{+}$ can be written as
$V=xx^{\ast}$ where $x\in \mathbb{C}^{d}$. One can see that $x$ is unique if we
restrict $x$ to the set $C_{+}^{d}=\{x=\left(  x_{1},x_{2},\ldots
,x_{d}\right)  \allowbreak:\allowbreak x_{1}\allowbreak \geq \allowbreak
0,\allowbreak$ $x_{j}\allowbreak \in \allowbreak \mathbb{C},\allowbreak$
$j\allowbreak=2,\allowbreak...,\allowbreak d\}$.

(b) Every $V\in \mathbb{H}_{d\left(  1\right)  }$ can be written as $V=\lambda
uu^{\ast}$ where $\lambda$ the eigenvalue of $V$ and $u$ is a unitary vector
in $\mathbb{C}^{d}$. In this representation the $d\times d$ matrix $uu^{\ast}$
is unique.
\end{remark}

\paragraph{Covariation of complex matrix semimartingales}

An $\mathbb{M}_{d\times q}$-valued process $X=\left \{  (x_{ij})(t)\right \}
_{t\geq0}$ is a matrix semimartingale if $x_{ij}(t)$ is a complex
semimartingale for each $i=1,...,d,j=1,...,q.$ Let $X=\left \{  (x_{ij}%
)(t)\right \}  _{t\geq0}$ $\in \mathbb{M}_{d\times q}$ and $Y=\left \{
(y_{ij})(t)\right \}  _{t\geq0}\in \mathbb{M}_{q\times r}$ be semimartingales.
Similar to the case of matrices with real entries in \cite{BNSe07}, we define
the matrix covariation of $X$ and $Y$ as the $\mathbb{M}_{d\times r}$-valued
process $\left[  X,Y\right]  :=\left \{  \left[  X,Y\right]  (t):t\geq
0\right \}  $ with entries%
\begin{equation}
\left[  X,Y\right]  _{ij}(t)=\sum \limits_{k=1}^{q}\left[  x_{ik}%
,y_{kj}\right]  (t)\text{,} \label{DefCov}%
\end{equation}
where $\left[  x_{ik},y_{kj}\right]  (t)$ is the covariation of the
$\mathbb{C}$-valued semimartingales $\left \{  x_{ik}(t)\right \}  _{t\geq0}$
and $\left \{  x_{kj}(t)\right \}  _{t\geq0}$; see \cite[pp 83]{Pr04}. One has
the decomposition into a continuous part and a pure jump part as follows%
\begin{equation}
\left[  X,Y\right]  (t)=\left[  X^{c},Y^{c}\right]  (t)+\sum_{s\leq t}\left(
\Delta X(s)\right)  \left(  \Delta Y(s)\right)  \text{,} \label{ForCov}%
\end{equation}
where $\left[  X^{c},Y^{c}\right]  _{ij}(t):=\sum \nolimits_{k=1}^{q}\left[
x_{ik}^{c},y_{kj}^{c}\right]  (t).$ We recall that for any semimartingale $x$,
the process $x^{c}$ is the $a.s.$ unique continuous local martingale $m$ such
that $\left[  x-m\right]  $ is purely discontinuous.

We will use the facts that $\left[  X\right]  =\left[  X,X^{\ast}\right]  $ is
a nonnegative definite $d\times d$ matrix, that $\left[  X,Y\right]  ^{\top
}=\left[  Y^{\top},X^{\top}\right]  $ and that for any nonrandom matrices
$A\in \mathbb{M}_{m\times d},C\in \mathbb{M}_{r\times n}$ and semimartingales
$X\in \mathbb{M}_{d\times q},Y\in \mathbb{M}_{q\times r}$,%
\begin{equation}
\left[  AX,YC\right]  =A\left[  X,Y\right]  C\text{.} \label{CovBil}%
\end{equation}

The natural example of a continuous semimartingale is the standard complex
$d\times q$ matrix Brownian motion $B=\left \{  B(t)\right \}  _{t\geq
0}=\left \{  b_{jl}(t)\right \}  _{t\geq0}$ consisting of independent
$\mathbb{C}$-valued Brownian motions $b_{jl}(t)=\operatorname{Re}%
(b_{jl}(t))+\mathrm{i}\operatorname{Im}(b_{jl}(t))$ where $\operatorname{Re}%
(b_{jl}(t)),\operatorname{Im}(b_{jl}(t))$ are independent one-dimensional
Brownian motions with common variance $t/2$. Then we have $\left[  B,B^{\ast
}\right]  ^{ij}(t)=\sum \nolimits_{k=1}^{q}\left[  b_{ik},\overline{b}%
_{jk}\right]  (t)=qt\delta_{ij}$ and hence the matrix quadratic variation of
$B$ is given by the $d\times d$ matrix process:%
\begin{equation}
\left[  B,B^{\ast}\right]  (t)=qt\mathrm{I}_{d}\text{.} \label{CovCB}%
\end{equation}
The case $q=1$ corresponds to the $\mathbb{C}^{d}$-valued standard Brownian
motion $B$.$\ $We observe this corresponds to $\left[  B,B^{\ast}\right]
_{t}=t\mathrm{I}_{d}$ instead of the common $2t\mathrm{I}_{d}$ used in the literature.

Other examples of complex matrix semimartingales are L\'{e}vy processes
considered next.

\paragraph{Complex matrix L\'{e}vy processes}

An infinitely divisible random matrix $M$ in $\mathbb{M}_{d\times q}$ is
characterized by the L\'{e}vy-Khintchine representation of its Fourier
transform $\mathbb{E}\mathrm{e}^{\mathrm{itr}(\Theta^{\ast}M)}\allowbreak
\ =\  \allowbreak \exp(\psi(\Theta))$ with Laplace exponent
\begin{equation}
\psi(\Theta)={}\mathrm{itr}(\Theta^{\ast}\Psi \text{ }){}-{}\frac{1}%
{2}\mathrm{tr}\left(  \Theta^{\ast}\mathcal{A}\Theta^{\ast}\right)  {}+{}%
\int_{\mathbb{M}_{d\times q}}\left(  \mathrm{e}^{\mathrm{itr}(\Theta^{\ast}%
\xi)}{}-1{}-\mathrm{i}\frac{\mathrm{tr}(\Theta^{\ast}\xi)}{1+\left \Vert
\xi \right \Vert ^{2}}{}\right)  \nu(\mathrm{d}\xi),\  \Theta \in \mathbb{M}%
_{d\times q}, \label{LKRgen}%
\end{equation}
where $\mathcal{A}:\mathbb{M}_{q\times d}\rightarrow \mathbb{M}_{d\times q}$ is
a positive symmetric linear operator $($i.e. $\mathrm{tr}\left(  \Phi^{\ast
}\mathcal{A}\Phi^{\ast}\right)  \geq0$ for $\Phi \in \mathbb{M}_{d\times q}$ and
$\mathrm{tr}\left(  \Theta_{2}^{\ast}\mathcal{A}\Theta_{1}^{\ast}\right)
=\mathrm{tr}\left(  \Theta_{1}^{\ast}\mathcal{A}\Theta_{2}^{\ast}\right)  $
for $\Theta_{1},\Theta_{2}\in \mathbb{M}_{d\times q})$, $\nu$ is a measure on
$\mathbb{M}_{d\times q}$ (the L\'{e}vy measure) satisfying $\nu(\{0\})=0$ and
$\int_{\mathbb{M}_{d\times q}}(1\wedge \left \Vert x\right \Vert ^{2}%
)\nu(\mathrm{d}x)<\infty$, and $\Psi \in \mathbb{M}_{d\times q}$. The triplet
$(\mathcal{A},\nu,\Psi)$ uniquely determines the distribution of $M$.

\begin{remark}
\label{ObsGaPart}The notation $\mathcal{A}\Theta^{\ast}$ means the linear
operator $\mathcal{A}$ from $\mathbb{M}_{q\times d}$ to $\mathbb{M}_{d\times
q}$ acting on $\Theta^{\ast}\in \mathbb{M}_{q\times d}$. Some interesting
examples of $\mathcal{A}$ and its corresponding matrix Gaussian distributions are:

(a) $\mathcal{A}\Theta^{\ast}=\Theta$. This corresponds to a Gaussian matrix
distribution invariant under left and right unitary transformations in
$\mathbb{U}_{d}$ and $\mathbb{U}_{q}$, respectively.

(b) $\mathcal{A}\Theta^{\ast}=\Sigma_{1}\Theta \Sigma_{2}$ for $\Sigma_{1}\in$
$\overline{\mathbb{H}}_{d}^{+}$ and $\Sigma_{2}\in \overline{\mathbb{H}}%
_{q}^{+}$. In this case the corresponding matrix Gaussian distribution is
denoted by $\mathrm{N}_{d\times q}(0,\Sigma_{1}\otimes \Sigma_{2})$ \ and
$\Sigma_{1}\otimes \Sigma_{2}$ is called a Kronecker covariance. It holds that
if $N$ has the distribution $\mathrm{N}_{d\times q}(0,\mathrm{I}_{d}{}%
\otimes \mathrm{I}_{q})$, then $\Sigma_{1}^{1/2}N\Sigma_{2}^{1/2}$ has
distribution $\mathrm{N}_{d\times q}(0,\Sigma_{1}\otimes \Sigma_{2})$.

(c) When $q=d$, $\mathcal{A}\Theta^{\ast}=$ $\mathrm{tr}(\Theta)\mathrm{I}%
_{d}$ is the covariance operator of the Gaussian random matrix $g\mathrm{I}%
_{d}$ where $g$ is a one-dimensional random variable with a standard Gaussian distribution.
\end{remark}

Let $\mathbb{S}_{d\times q}$ be the unit sphere of $\mathbb{M}_{d\times q}$
and let $\mathbb{M}_{d\times q}^{0}=\mathbb{M}_{d\times q}\backslash \{0\}$. If
$\nu$ is a L\'{e}vy measure on $\mathbb{M}_{d\times q}$, then there are a
measure $\lambda$ on $\mathbb{S}_{d\times q}$ with $\lambda(\mathbb{S}%
_{d\times q})\geq0$ and a measure $\nu_{\xi}$ for each $\xi \in \mathbb{S}%
_{d\times q}$ with $\nu_{\xi}((0,\infty))>0$ such that
\[
\nu(E)=\int_{\mathbb{S}_{d\times q}}\lambda(\mathrm{d}\xi)\int_{(0,\infty
)}1_{E}(u\xi)\nu_{\xi}(\mathrm{d}u),\qquad E\in \mathcal{B}(\mathbb{M}_{d\times
q}^{0}).
\]
We call $(\lambda,\nu_{\xi})$ a polar decomposition of $\nu$. When $d=q=1$,
$\nu$ is a L\'{e}vy measure on $\mathbb{R}$ and $\lambda$ is a measure in the
unit sphere $\mathbb{S}_{1\times1}=\left \{  -1,1\right \}  $ of $\mathbb{R}$.

Any $\mathbb{M}_{d\times q}$-valued L\'{e}vy process $L=\left \{  L(t)\right \}
_{t\geq0}$ with triplet $(\mathcal{A},\nu,\Psi)$ is a semimartingale with the
L\'{e}vy-It\^{o} decomposition
\begin{equation}
L(t)=t\Psi+B_{\mathcal{A}}(t)+\int_{[0,t]}\int_{\left \Vert V\right \Vert \leq
1}V\widetilde{J}_{L}(\mathrm{d}s\mathrm{,d}V)+\int_{[0,t]}\int_{\left \Vert
V\right \Vert >1}VJ_{L}(\mathrm{d}s,\mathrm{d}V)\text{, }t\geq0, \label{LID}%
\end{equation}
where:

(a) $\left \{  B_{\mathcal{A}}(t)\right \}  _{t\geq0}$ is a $\mathbb{M}_{d\times
q}$-valued Brownian motion with covariance $\mathcal{A}$, i.e. it is a
L\'{e}vy process with continuous sample paths (a.s.) and each $B_{\mathcal{A}%
}(t)$ is centered Gaussian with
\[
\mathbb{E}\left \{  \mathrm{tr(}\Theta_{1}^{\ast}B_{\mathcal{A}}(t){}%
)\mathrm{tr}\left(  \Theta_{2}^{\ast}B_{\mathcal{A}}(s){}\right)  {}\right \}
=\min(s,t)\mathrm{tr}\left(  \Theta_{1}^{\ast}\mathcal{A}\Theta_{2}^{\ast
}\right)  {}\text{for each }\Theta_{1},\Theta_{2}\in \mathbb{M}_{d\times q},
\]

(b) $J_{L}(\cdot,\cdot)$ is the Poisson random measure of jumps on
$[0,\infty)\times \mathbb{M}_{d\times q}^{0}$. That is, $J_{L}(t,E)=\# \{(0\leq
s\leq t:\allowbreak \Delta L_{s}\in E\},$ $E$ $\in \mathbb{M}_{d\times q}^{0},$
with intensity measure $Leb\otimes \nu$, and independent of $\left \{
B_{\mathcal{A}}(t)\right \}  _{t\geq0}$,

(c) $\widetilde{J}_{L}$ is the compensator measure of $J_{L}$, i.e.
\[
\widetilde{J}_{L}(\mathrm{d}t,\mathrm{d}V)=J_{L}(\mathrm{d}t,\mathrm{d}%
V)-\mathrm{d}t\nu(\mathrm{d}V);
\]
see for example \cite{Ap07} for the most general case of L\'{e}vy processes
with values in infinite dimensional Banach spaces.

An $\mathbb{M}_{d\times q}$-valued L\'{e}vy process $L=\left \{  L(t)\right \}
_{t\geq0}$ has bounded variation if and only if its L\'{e}vy-It\^{o}
decomposition takes the form
\begin{equation}
L(t)=t\Psi_{0}+\int_{[0,t]}\int_{\mathbb{M}_{d\times q}^{0}}VJ_{L}%
(\mathrm{d}s,\mathrm{d}V)=t\Psi_{0}+\sum_{s\leq t}\Delta L(s)\text{, }t\geq0,
\label{LIFV}%
\end{equation}
where $\Psi_{0}=$ $\Psi-\int_{\left \Vert V\right \Vert \leq1}V\nu
(\mathrm{d}V).$

The matrix quadratic variation (\ref{ForCov}) of $L$ is given by the
$\overline{\mathbb{H}}_{d}^{+}$-valued process
\begin{equation}
\lbrack L](t)=\left[  B_{\mathcal{A}},B_{\mathcal{A}}^{\ast}\right]
(t)+\int_{[0,t]}\int_{\mathbb{M}_{d\times q}^{0}}VV^{\ast}J_{L}(\mathrm{d}%
s,\mathrm{d}V)=\left[  B_{\mathcal{A}},B_{\mathcal{A}}^{\ast}\right]
(t)\mathcal{+}\sum_{s\leq t}\Delta L(s)\Delta L(s)^{\ast}. \label{QVLP}%
\end{equation}

In\ Section 3 we prove a partial converse of the last result in the case
$q=1.$

\begin{remark}
\label{ObsGaPartqv} On the lines of Remark \ref{ObsGaPart} we have the
following observations for the quadratic variation of the continuous part in
(\ref{QVLP}):

(a) When $\mathcal{A}\Theta^{\ast}=\Theta,$ $\left[  B_{\mathcal{A}%
},B_{\mathcal{A}}^{\ast}\right]  (t)=qt\mathrm{I}_{d}$. This follows from
(\ref{CovCB}) since $B_{\mathcal{A}}(t)$ is a standard complex $d\times q$
matrix Brownian motion.

(b) When $\mathcal{A}\Theta^{\ast}=\Sigma_{1}\Theta \Sigma_{2}$ for $\Sigma
_{1}\in$ $\overline{\mathbb{H}}_{d}^{+}$ and $\Sigma_{2}\in \overline
{\mathbb{H}}_{q}^{+}$, we have $B_{\mathcal{A}}(t)=\Sigma_{1}^{1/2}%
B(t)\Sigma_{2}^{1/2}$ where $B=\left \{  B(t)\right \}  _{t\geq0}$ is a standard
complex $d\times q$ matrix Brownian motion. Then, using (\ref{CovBil}) we
have
\[
\left[  B_{\mathcal{A}},B_{\mathcal{A}}^{\ast}\right]  (t)=\left[  \Sigma
_{1}^{1/2}B\Sigma_{2}^{1/2},\Sigma_{2}^{1/2}B^{\ast}\Sigma_{1}^{1/2}\right]
(t)=\Sigma_{1}^{1/2}\left[  B\Sigma_{2}^{1/2},\Sigma_{2}^{1/2}B^{\ast}\right]
(t)\Sigma_{1}^{1/2}=t\mathrm{tr}(\Sigma_{2})\Sigma_{1},
\]
where we have also used the easily checked fact $\left[  B\Sigma_{2}%
^{1/2},\Sigma_{2}^{1/2}B^{\ast}\right]  (t)=t\mathrm{tr}(\Sigma_{2})I_{d}$.

(c) When $q=d$ and $\mathcal{A}\Theta^{\ast}=$ $\mathrm{tr}(\Theta
)\mathrm{I}_{d}$, we have $\left[  B_{\mathcal{A}},B_{\mathcal{A}}^{\ast
}\right]  (t)=t\mathrm{I}_{d}$ since $B_{\mathcal{A}}(t)=b(t)\mathrm{I}_{d}$
where $b=\left \{  b(t)\right \}  _{t\geq0}$ is a one-dimensional Brownian motion.
\end{remark}

The extension of the notion of a real subordinator to the matrix case relies
on cones. A cone $K$ is a nonempty, closed, convex subset of $\mathbb{M}%
_{d\times q}$ such that if $A\in K$ and $\alpha \geq0$ imply $\alpha A\in K$. A
cone $K$ determines a partial order in $\mathbb{M}_{d\times q}$ by defining
$V_{1}\leq_{K}V_{2}$ for $V_{1},V_{2}\in \mathbb{M}_{d\times q}$ whenever
$V_{2}-V_{1}\in K$. A $\mathbb{M}_{d\times q}$-valued L\'{e}vy process
$L=\left \{  L(t)\right \}  _{t\geq0}$ is $K$- increasing if $L(t_{1})\leq
_{K}L(t_{2})$ for every $t_{1}<t_{2}$ almost surely. A $K$-increasing L\'{e}vy
process with values in $\mathbb{M}_{d\times q}$ is called a matrix
subordinator. It is easy to see that if $L=\left \{  L(t)\right \}  _{t\geq0}$
is a L\'{e}vy process in $\mathbb{M}_{d\times q}$ then $L$ is a subordinator
if and only if $L$ takes values in $K$. In this sense the matrix quadratic
variation L\'{e}vy process in (\ref{QVLP}) with values in the cone
$\overline{\mathbb{H}}_{d}^{+}$ is a matrix subordinator$.$

\paragraph{Approximation of L\'{e}vy processes}

The following are useful results on the sample path approximation of complex
matrix L\'{e}vy processes; see \cite[Th 15.17]{Ka} and \cite[Th. 8.7]{Sato1}.
They follow from their corresponding real vector case by the usual
identification of $\mathbb{M}_{d\times q}\rightarrow$ $\mathbb{R}^{2dq}$ via
$A\rightarrow \mathrm{vec}(A),$ $A\in$ $\mathbb{M}_{d\times q}$ and the fact
that $\mathrm{tr}\left(  A^{\ast}B\right)  =\mathrm{vec}(A)^{\ast}%
\mathrm{vec}(B)$, where $\mathrm{vec}(A)$ is the $dq$ column complex vector
obtained by stacking the columns of $A$ one down each other.

\begin{proposition}
\label{convproc}Let $L$ and $L^{n}$ $n=1,2,...$ be complex matrix L\'{e}vy
processes in $\mathbb{M}_{d\times q}$ with $L^{n}(1)\allowbreak \overset
{\mathcal{L}}{\rightarrow}L(1)$. Then there exist processes $\tilde{L}^{n}$
with the same distribution that $L^{n}$ such that%
\[
\sup_{0\leq s\leq t}\left \vert \tilde{L}^{n}(s)-L(s)\right \vert \overset{\Pr
}{\longrightarrow}0,\quad \forall t\geq0.
\]

\end{proposition}

\begin{proposition}
\label{convternas}Let $M^{n},n=1,2,...$ be infinitely divisible random
matrices in $\mathbb{M}_{d\times q}$ with triplet $(\mathcal{A}^{n},\nu
^{n},\Psi^{n})$. Let $M$ be a random matrix in $\mathbb{M}_{d\times q}$. Then
$M^{n}\overset{\mathcal{L}}{\rightarrow}M$ if and only if $M$ is infinitely
divisible whose triplet $(\mathcal{A},\nu,\Psi)$ satisfies the following three
conditions:\newline a) If $f:\mathbb{M}_{d\times q}\rightarrow \mathbb{M}%
_{d\times q}$ is bounded and continuous function vanishing in a neighborhood
of $0$ then%
\[
\lim_{n\rightarrow \infty}\int \nolimits_{\mathbb{M}_{d\times q}}f(\xi)\nu
^{n}(\mathrm{d}\xi)=\int \nolimits_{\mathbb{M}_{d\times q}}f(\xi)\nu
(\mathrm{d}\xi)\text{.}%
\]
b) Define the positive symmetric operator $\mathcal{A}^{n,\epsilon}%
:\mathbb{M}_{q\times d}\rightarrow \mathbb{M}_{d\times q}$ by%
\[
\mathrm{tr}\left(  \Theta^{\ast}\mathcal{A}^{n,\epsilon}\Theta^{\ast}\right)
=\mathrm{tr}\left(  \Theta^{\ast}\mathcal{A}^{n}\Theta^{\ast}\right)
+\int \nolimits_{\left \Vert \xi \right \Vert \leq \varepsilon}\left \vert
\mathrm{tr}\left(  \Theta^{\ast}\xi \right)  \right \vert ^{2}\nu_{n}%
(\mathrm{d}\xi)\quad \text{for }\Theta \in \mathbb{M}_{d\times q}\text{.}%
\]
Then
\[
\lim_{\varepsilon \downarrow0}\limsup_{n\rightarrow \infty}\left \vert
\mathrm{tr}\left(  \Theta^{\ast}\mathcal{A}^{n,\epsilon}\Theta^{\ast}\right)
-\mathrm{tr}\left(  \Theta^{\ast}\mathcal{A}\Theta^{\ast}\right)  \right \vert
=0,\quad \text{for }\Theta \in \mathbb{M}_{d\times q}\text{.}%
\]
c) $\Psi^{n}\rightarrow \Psi$.
\end{proposition}

\section{BGCD random matrix ensembles}

We now consider the matrix L\'{e}vy processes associated to the BGCD matrix
ensembles $(M_{d})_{d\geq1}$ mentioned in the introduction.

When $\mu$ is the standard Gaussian distribution, $M_{d}$ is a Gaussian
unitary invariant random matrix, $\Lambda(\mu)$ is the semicircle distribution
and $\left \{  M_{d}(t)\right \}  _{t\geq0}$ is the Hermitian matrix valued
process given by $M_{d}(t)=\left(  1/\sqrt{d+1}\right)  (B(t)+dg(t)\mathrm{I}%
_{d})$ where $B\left(  t\right)  $ is a $d\times d$ Hermitian matrix Brownian
motion independent of the one-dimensional Brownian motion $g(t);$ see
\cite[Remark 3.5]{BG}.

Likewise, if $\mu$ is the Poisson distribution with parameter $\lambda>0,$
$\left \{  M_{d}(t)\right \}  _{t\geq0}$ is the $d\times d$ matrix compound
Poisson process $M_{d}(t)=\sum_{k=1}^{N(t)}u_{k}^{d}u_{k}^{d\ast}$ where
$\left \{  u_{k}^{d}\right \}  _{k\geq1}$ is a sequence of independent uniformly
distributed random vectors on the unit sphere of $\mathbb{C}^{d}$ independent
of the Poisson process $\left \{  N(t)\right \}  _{t\geq0}$, and $\Lambda(\mu)$
is the Marchenko-Pastur distribution of parameter $\lambda>0$; see
\cite[Remark 3.2]{BG}. We observe that in this case $\left \{  M_{d}%
(t)\right \}  _{t\geq0}$ is a matrix covariation (quadratic) process rather
than a covariance matrix process as in the Wishart or other empirical
covariance processes.

Proposition \ref{polar} below collects computations in \cite{BG}, \cite{CD}
and \cite{DRA} to summarize the L\'{e}vy triplet of a general BGCD matrix
ensemble in an explicit manner. Let $\nu|_{(0,\infty)}$ and $\nu
|_{(-\infty,0)}$ denote the corresponding restrictions to $\left(
0,+\infty \right)  $ and $\left(  -\infty,0\right)  $ for any L\'{e}vy measure
$\nu$, respectively.

\begin{proposition}
\label{polar}Let $\mu \ $be an infinitely divisible distribution in
$\mathbb{R}$ with L\'{e}vy triplet $(a^{2}$,$\mathcal{\psi},\nu)$ and let
$(M_{d})_{d\geq1}$ be a BGCD matrix ensemble for $\Lambda(\mu)$. Then, for
each $d\geq1$ $M_{d}$ has the L\'{e}vy-Khintchine representation
(\ref{LKRgen}) with L\'{e}vy triplet $(\mathcal{A}_{d},\Psi_{d},\nu_{d})$ where

a) $\Psi_{d}=\mathcal{\psi}\mathrm{I}_{d}$

b)
\begin{equation}
\mathcal{A}_{d}\Theta=a^{2}\frac{1}{d+1}(\Theta+\mathrm{tr}(\Theta
)\mathrm{I}_{d}),\quad \Theta \in \mathbb{H}_{d}. \label{GPBGCD}%
\end{equation}

c)
\begin{equation}
\nu_{d}\left(  E\right)  =d\int_{\mathbb{S}(\mathbb{H}_{d(1)})}\int
_{0}^{\infty}1_{E}\left(  rV\right)  \nu_{V}\left(  \mathrm{d}r\right)
\Pi \left(  \mathrm{d}V\right)  \text{,\quad}E\in \mathcal{B}\left(
\mathbb{H}_{d}\backslash \left \{  0\right \}  \right)  \text{,} \label{PDBGCD}%
\end{equation}
where $\nu_{V}=\nu|_{(0,\infty)}$ or $\nu|_{(-\infty,0)}$ according to
$V\geq0$ or\ $V\leq0$ and $\Pi$ is a measure on $\mathbb{S}(\mathbb{H}%
_{d(1)})$ such that
\begin{equation}
\Pi \left(  D\right)  =\int \limits_{\mathbb{S}(\mathbb{H}_{d(1)})\cap
\overline{\mathbb{H}}_{d}^{+}}\int \limits_{\left \{  -1,1\right \}  }%
1_{D}\left(  tV\right)  \lambda \left(  \mathrm{d}t\right)  \omega_{d}\left(
\mathrm{d}V\right)  \text{,\quad}D\in \mathcal{B}\left(  \mathbb{S}%
(\mathbb{H}_{d(1)})\right)  \text{,} \label{pi}%
\end{equation}
where $\lambda$ is the spherical measure of $\nu$ and $\omega_{d}$ is the
probability measure on $\mathbb{S}(\mathbb{H}_{d(1)})\cap \overline{\mathbb{H}%
}_{d}^{+}$ induced by the transformation $u\rightarrow V=uu^{\ast}$, where $u$
is a uniformly distributed column random vector in the unit sphere of
$\mathbb{C}^{d}$.
\end{proposition}

\begin{proof}
(a) It follows from the first term in the L\'{e}vy exponent of $M_{d}$ in page
$635$ of \cite{CD}, where the notation $\Lambda_{d}(\mu)$ is used for the
distribution of $M_{d}$. For (b), the form of the covariance operator
$\mathcal{A}_{d}$ was implicitly computed in the first example in Section II.C
of \cite{CD}. Finally, the polar decomposition of the L\'{e}vy measure
(\ref{PDBGCD}) was found in \cite{DRA}.
\end{proof}

The L\'{e}vy-It\^{o} decomposition of the L\'{e}vy process associated to the
BGCD model $M_{d}$ is given by
\begin{equation}
M_{d}(t)=\mathcal{\psi}td\mathrm{I}_{d}+B_{\mathcal{A}_{d}}(t)+\int
_{[0,t]}\int_{\left \{  \left \Vert V\right \Vert \leq1\right \}  \cap
\mathbb{H}_{d(1)}}V\widetilde{J}_{d}(\mathrm{d}s\mathrm{,d}V)+\int_{[0,t]}%
\int_{\left \{  \left \Vert V\right \Vert >1\right \}  \cap \mathbb{H}_{d(1)}%
}VJ_{d}(\mathrm{d}s,\mathrm{d}V)\text{,} \label{LIDbgcd}%
\end{equation}
where $t\geq0$, $\mathcal{A}_{d}\Theta=a^{2}\frac{1}{d+1}(\Theta
+\mathrm{tr}(\Theta)\mathrm{I}_{d})$, $J_{d}(t,E)=\# \left \{  0\leq s\leq
t:\Delta M_{d}(s)\in E\right \}  =J_{d}(t,E\cap \mathbb{H}_{d(1)})$ for any
measurable $E$ $\in \mathbb{H}_{d}\backslash \{0\}$. Its quadratic variation is
obtained by (\ref{QVLP}) as the matrix subordinator%
\begin{subequations}
\begin{equation}
\left[  M_{d}\right]  (t)=a^{2}t\mathrm{I}_{d}+\int_{[0,t]}\int_{\mathbb{H}%
_{d(1)}\backslash \{0\}}VV^{\ast}J_{d}(\mathrm{d}s,\mathrm{d}V)=a^{2}%
t\mathrm{I}_{d}+\sum_{s\leq t}\Delta M_{d}(s)\left(  \Delta M_{d}(s)\right)
^{\ast}.\nonumber
\end{equation}

\end{subequations}
\begin{remark}
It is possible to obtain BGCD models of symmetric random matrices rather than
Hermitian. Indeed, slight changes in the proof of \cite[Theorem 3.1]{BG} give
for each $d\geq1$, a $d\times d$ real symmetric random matrix $M_{d}$ with
orthogonal invariant infinitely divisible matrix distribution. The asymptotic
spectral distribution of the corresponding Hermitian and symmetric ensembles
is the same, similarly as the semicircle distribution is the asymptotic
spectral distribution for the Gaussian Unitary Ensemble and Gaussian
Orthogonal Ensemble.
\end{remark}

\section{Bounded variation case}

It is well known that the quadratic variation of a one-dimensional L\'{e}vy
process is a subordinator, see \cite[Example 8.5]{CT}. The following result
gives a converse and a generalization to matrix subordinators with rank one
jumps. The one dimensional case is given in \cite[Lemma 6.5]{Se10}.

\begin{theorem}
\label{Sub} Let $L_{d}=\left \{  L_{d}(t):t\geq0\right \}  $ be a L\'{e}vy
process in $\overline{\mathbb{H}}_{d}^{+}$ whose jumps are of rank one almost
surely. Then there exists a L\'{e}vy process $X=\left \{  X(t):t\geq0\right \}
$ in $\mathbb{C}^{d}$ such that $L_{d}(t)=\left[  X\right]  (t)$.
\end{theorem}

\begin{proof}
We construct $X$ as a L\'{e}vy-It\^{o} decomposition realization. Using
(\ref{LIFV}), for each $d\geq1$, $L_{d}$ is an $\overline{\mathbb{H}}_{d}^{+}%
$-process of bounded variation with L\'{e}vy-It\^{o} decomposition%
\[
L_{d}(t)=\Psi_{0}t+\int \nolimits_{\left[  0,t\right]  }\int_{\mathbb{H}%
_{d\left(  1\right)  }^{+}\backslash \{0\}}VJ_{L_{d}}(\mathrm{d}s,\mathrm{d}%
V)\text{, }t\geq0\text{,}%
\]
where $\Psi_{0}\in \mathbb{H}_{d}^{+}$ and $J_{L_{d}}$ is the Poisson random
measure of $L_{d}$. Let $Leb\otimes \nu_{L_{d}}$ denote the intensity measure
of $J_{L_{d}}$.

Consider the cone $C_{+}^{d}=\left \{  x=\left(  x_{1},x_{2},\ldots
,x_{d}\right)  :x_{1}\geq0,\text{ }x_{j}\in \mathbb{C},\text{ }%
j=2,...,d\right \}  $ and let $\varphi_{+}:\mathbb{R}_{+}\times \mathbb{H}%
_{d(1)}^{+}\rightarrow \mathbb{R}_{+}\times C_{+}^{d}$ be defined as
$\varphi_{+}\left(  t,V\right)  =(t,x)$ where $V=xx^{\ast}$ and $x\in
C_{+}^{d}$. Let $\overline{\varphi}_{+}:\mathbb{H}_{d(1)}^{+}\rightarrow
C_{+}^{d}$ be defined by $\overline{\varphi}_{+}\left(  V\right)  =x$ for
$V=xx^{\ast}$ and $x\in C_{+}^{d}$. By Remark \ref{Decomp} (a) the functions
$\varphi_{+}$ and $\overline{\varphi}_{+}$ are well defined.

Let us define $J(\mathrm{d}s,\mathrm{d}x)=\left(  J_{L_{d}}\circ \varphi
_{+}^{-1}\right)  \left(  \mathrm{d}s,\mathrm{d}x\right)  $ the random measure
induced by the transformation $\varphi_{+}$ which is a Poisson random measure
on $\mathbb{R}_{+}\times C_{+}^{d}$. Observe that $\mathbb{E}\left[
J(t,F)\right]  =\mathbb{E}\left[  J_{L_{d}}\circ \varphi_{+}^{-1}\left(
\left \{  t\right \}  \times F\right)  \right]  =t\nu_{L_{d}}\left(
\overline{\varphi}_{+}\left(  F\right)  \right)  =t\left(  \nu_{L_{d}}%
\circ \overline{\varphi}_{+}^{-1}\right)  \left(  F\right)  $ for
$F\allowbreak \in \allowbreak \mathcal{B(}\allowbreak C_{+}^{d}\allowbreak
\backslash \left \{  0\right \}  )$. Let us denote $\nu=\nu_{L_{d}}\circ
\overline{\varphi}_{+}^{-1}$ which is a L\'{e}vy measure on $C_{+}^{d}$ since%
\[
\int_{C_{+}^{d}\backslash \left \{  0\right \}  }\left(  1\wedge \left \vert
x\right \vert ^{2}\right)  \nu(\mathrm{d}x)=\int_{C_{+}^{d}\backslash \left \{
0\right \}  }\left(  1\wedge \left \vert x\right \vert ^{2}\right)  \nu_{L_{d}%
}\circ \overline{\varphi}_{+}^{-1}(\mathrm{d}x)
\]%
\[
=\int_{C_{+}^{d}\backslash \left \{  0\right \}  }\left(  1\wedge \mathrm{tr}%
\left(  xx^{\ast}\right)  \right)  \nu_{L_{d}}\circ \overline{\varphi}_{+}%
^{-1}(\mathrm{d}x)=\int_{\mathbb{H}_{d(1)}^{+}\backslash \left \{  0\right \}
}\left(  1\wedge \mathrm{tr}\left(  V\right)  \right)  \left(  \nu_{L_{d}}%
\circ \overline{\varphi}_{+}^{-1}\right)  \circ f^{-1}(\mathrm{d}V)
\]%
\[
=\int_{\mathbb{H}_{d(1)}^{+}\backslash \left \{  0\right \}  }\left(
1\wedge \mathrm{tr}\left(  V\right)  \right)  \nu_{L_{d}}(\mathrm{d}%
V)<\infty \text{,}%
\]
where $\left(  \nu_{L_{d}}\circ \overline{\varphi}_{+}^{-1}\right)  \circ
f^{-1}=\nu_{L_{d}},$ with $f\left(  x\right)  =xx^{\ast}$ and we have used
that $\mathrm{tr}\left(  V\right)  \leq \alpha \left \Vert V\right \Vert $ for
some constant $\alpha>0$. Thus $Leb\otimes \nu$ is the intensity measure of the
Poisson random measure $J$.

Let us take the L\'{e}vy process in $\mathbb{C}^{d}$%
\begin{equation}
X(t)=\left \vert \Psi_{0}\right \vert ^{1/2}B_{I}(t)+\int \nolimits_{\left[
0,t\right]  }\int_{\mathbb{C}^{d}\cap \{ \left \vert x\right \vert \leq
1\}}x\widetilde{J}(\mathrm{d}s,\mathrm{d}x)+\int \nolimits_{\left[  0,t\right]
}\int_{\mathbb{C}^{d}\cap \{ \left \vert x\right \vert >1\}}xJ(\mathrm{d}%
s,\mathrm{d}x)\text{, }t\geq0\text{,} \label{LIX}%
\end{equation}
where $B_{I}$ is a $\mathbb{C}^{d}$-valued standard Brownian motion with
quadratic variation $tI_{d}$, (i.e. (\ref{CovCB}) with $q=1$). Thus the
quadratic variation of $X$ is given by%
\[
\left[  X\right]  (t)=\left[  \left \vert \Psi_{0}\right \vert ^{1/2}B_{I}%
,B_{I}^{\ast}\left \vert \Psi_{0}\right \vert ^{1/2}\right]  (t)+\int
\nolimits_{\left[  0,t\right]  }\int_{\mathbb{C}^{d}\backslash \{0\}}xx^{\ast
}J(\mathrm{d}s,\mathrm{d}x)
\]%
\[
=\Psi_{0}t+\int \nolimits_{\left[  0,t\right]  }\int_{\mathbb{C}^{d}%
\backslash \{0\}}xx^{\ast}J_{L_{d}}\circ \varphi_{+}^{-1}(\mathrm{d}%
s,\mathrm{d}x)=\Psi_{0}t+\int \nolimits_{\left[  0,t\right]  }\int
_{\mathbb{H}_{d(1)}^{+}\backslash \left \{  0\right \}  }VJ_{L_{d}}\circ
\varphi_{+}^{-1}\circ h^{-1}(\mathrm{d}s,\mathrm{d}V)
\]%
\[
=\Psi_{0}t+\int \nolimits_{\left[  0,t\right]  }\int_{\mathbb{H}_{d(1)}%
^{+}\backslash \left \{  0\right \}  }VJ_{L_{d}}(\mathrm{d}s,\mathrm{d}%
V)=L_{d}(t),
\]
where $J_{L_{d}}\circ \varphi_{+}^{-1}\circ h^{-1}=J_{L_{d}},$ with $h\left(
t,x\right)  =\left(  t,xx^{\ast}\right)  .$
\end{proof}

\smallskip

For the general bounded variation case we have the following Wiener-Hopf type decomposition.

\begin{theorem}
\label{Bdv}Let $L_{d}=\left \{  L_{d}(t):t\geq0\right \}  $ be a L\'{e}vy
process in $\mathbb{H}_{d}$ of bounded variation whose jumps are of rank one
almost surely. Then there exist L\'{e}vy processes $X=\left \{  X(t):t\geq
0\right \}  $ and $Y=\left \{  Y(t):t\geq0\right \}  $ in $\mathbb{C}^{d}$ such
that
\begin{equation}
L_{d}(t)=\left[  X\right]  (t)-\left[  Y\right]  (t). \label{WHTD}%
\end{equation}
Moreover, $\left \{  \left[  X\right]  (t):t\geq0\right \}  $ and $\left \{
\left[  Y\right]  (t):t\geq0\right \}  $ are independent processes.
\end{theorem}

\begin{proof}
For each $d\geq1$, $L_{d}$ is an $\mathbb{H}_{d}$-process of bounded variation
with L\'{e}vy-It\^{o} decomposition%
\begin{equation}
L_{d}(t)=\Psi t+\int \nolimits_{\left[  0,t\right]  }\int_{\mathbb{H}%
_{d(1)}\backslash \{0\}}VJ_{L_{d}}(\mathrm{d}s,\mathrm{d}V)\text{, }%
t\geq0\text{,} \label{LIDBV}%
\end{equation}
where $\Psi \in \mathbb{H}_{d}$ and $J_{L_{d}}$ is the Poisson random measure of
$L_{d}$. Let $Leb\otimes \nu_{L_{d}}$ denote the intensity measure of
$J_{L_{d}}$.

First we prove that $L_{d}=L_{d}^{1}-L_{d}^{2}$ where $L_{d}^{1}$ and
$L_{d}^{2}$ are the L\'{e}vy processes in $\overline{\mathbb{H}}_{d}^{+}$
given by (\ref{Quad1decomp}) and (\ref{Quad2decomp}).\newline Every
$V\in \mathbb{H}_{d\left(  1\right)  }$ can be written as $V=\lambda uu^{\ast}$
where $\lambda$ the eigenvalue of $V$ and $u$ is a unitary vector in
$\mathbb{C}^{d}$. Let us define $\left \vert V\right \vert =\left \vert
\lambda \right \vert uu^{\ast}$ and $V^{+}=\lambda^{+}uu^{\ast}$, $V^{-}%
=\lambda^{-}uu^{\ast}$ where $\lambda^{+}=\lambda$ if $\lambda \geq0$ and
$\lambda^{-}=-\lambda$ if $\lambda<0$.

Let $\varphi_{+}:\mathbb{R}_{+}\times \mathbb{H}_{d(1)}\rightarrow
\mathbb{R}_{+}\times \mathbb{H}_{d(1)}^{+}$ and $\varphi_{-}:\mathbb{R}%
_{+}\times \mathbb{H}_{d(1)}\rightarrow \mathbb{R}_{+}\times \mathbb{H}%
_{d(1)}^{+}$ be defined as $\varphi_{+}\left(  t,V\right)  =(t,V^{+})$ and
$\varphi_{-}\left(  t,V\right)  =(t,V^{-})$ respectively. Let $\overline
{\varphi}_{+}:\mathbb{H}_{d(1)}\rightarrow \mathbb{H}_{d(1)}^{+}$ and
$\overline{\varphi}_{-}:\mathbb{H}_{d(1)}\rightarrow \mathbb{H}_{d(1)}^{+}$ be
defined as $\overline{\varphi}_{+}(V)=V^{+}$ and $\overline{\varphi}%
_{-}(V)=V^{-}$ respectively. By Remark \ref{Decomp} (b) the functions
$\varphi_{+},$ $\overline{\varphi}_{+},$ $\varphi_{-}$ and $\overline{\varphi
}_{-}$ are well defined and hence $V=\overline{\varphi}_{+}(V)-\overline
{\varphi}_{-}(V)$.

Let us define $J^{+}(\mathrm{d}s,\mathrm{d}x)=\left(  J_{L_{d}}\circ
\varphi_{+}^{-1}\right)  \left(  \mathrm{d}s,\mathrm{d}x\right)  $ and
$J^{-}(\mathrm{d}s,\mathrm{d}x)=\left(  J_{L_{d}}\circ \varphi_{-}^{-1}\right)
\left(  \mathrm{d}s,\mathrm{d}x\right)  $ the random measures induced by the
transformations $\varphi_{+}$ and $\varphi_{-}$ respectively, which are
Poisson random measures both on $\mathbb{R}_{+}\times \mathbb{H}_{d(1)}^{+}$.
Observe that $\mathbb{E}\left[  J^{+}(t,F)\right]  =\mathbb{E[}J_{L_{d}}%
\circ \allowbreak \varphi_{+}^{-1}(\allowbreak \left \{  t\right \}  \allowbreak
\times \allowbreak F)]=t\nu_{L_{d}}\left(  \overline{\varphi}_{+}^{-1}\left(
F\right)  \right)  =t\left(  \nu_{L_{d}}\circ \overline{\varphi}_{+}%
^{-1}\right)  \left(  F\right)  $ for $F\in \mathcal{B}\left(  \mathbb{H}%
_{d(1)}^{+}\backslash \left \{  0\right \}  \right)  $ and similarly
$\mathbb{E}\left[  J^{-}(t,F)\right]  =t\left(  \nu_{L_{d}}\circ
\overline{\varphi}_{-}^{-1}\right)  \left(  F\right)  $. Let us denote
$\nu_{L_{d}}^{+}=\nu_{L_{d}}\circ \overline{\varphi}_{+}^{-1}$ and $\nu_{L_{d}%
}^{-}=\nu_{L_{d}}\circ \overline{\varphi}_{-}^{-1}$. Note that\ $\nu_{L_{d}%
}^{+}$ is a L\'{e}vy measure on $\mathbb{H}_{d(1)}^{+}$ since%
\begin{align*}
\infty &  >\int_{\mathbb{H}_{d(1)}\backslash \left \{  0\right \}  }\left(
1\wedge \left \Vert V\right \Vert \right)  \nu_{L_{d}}(\mathrm{d}V)\geq
\int_{\mathbb{H}_{d(1)}\backslash \left \{  0\right \}  }\left(  1\wedge
\left \Vert \overline{\varphi}_{+}(V)\right \Vert \right)  \nu_{L_{d}%
}(\mathrm{d}V)\\
&  =\int_{\mathbb{H}_{d(1)}^{+}\backslash \left \{  0\right \}  }\left(
1\wedge \left \Vert W\right \Vert \right)  \nu_{L_{d}}^{+}(\mathrm{d}W)\text{.}%
\end{align*}
Hence $Leb\otimes \nu_{L_{d}}^{+}$ is the intensity measure of $J^{+}$.
Similarly, one can see that $Leb\otimes \nu_{L_{d}}^{-}$ is the intensity
measure of $J^{-}$.

There exist $\Psi^{+}$ and $\Psi^{-}$ in $\mathbb{H}_{d}^{+}$ such that
$\Psi=\Psi^{+}-\Psi^{-}$. Let us take the L\'{e}vy processes $X$ and $Y$ in
$\mathbb{C}^{d}$%
\[
X(t)=\left \vert \Psi^{+}\right \vert ^{1/2}B_{I}(t)+\int \nolimits_{\left[
0,t\right]  }\int_{\mathbb{C}^{d}\cap \{ \left \vert x\right \vert \leq
1\}}x\widetilde{J}^{+}(\mathrm{d}s,\mathrm{d}x)+\int \nolimits_{\left[
0,t\right]  }\int_{\mathbb{C}^{d}\cap \{ \left \vert x\right \vert >1\}}%
xJ^{+}(\mathrm{d}s,\mathrm{d}x)\text{, }t\geq0\text{,}%
\]%
\[
Y(t)=\left \vert \Psi^{-}\right \vert ^{1/2}B_{I}(t)+\int \nolimits_{\left[
0,t\right]  }\int_{\mathbb{C}^{d}\cap \{ \left \vert x\right \vert \leq
1\}}x\widetilde{J}^{-}(\mathrm{d}s,\mathrm{d}x)+\int \nolimits_{\left[
0,t\right]  }\int_{\mathbb{C}^{d}\cap \{ \left \vert x\right \vert >1\}}%
xJ^{-}(\mathrm{d}s,\mathrm{d}x)\text{, }t\geq0\text{,}%
\]
where $B_{I}$ is a $\mathbb{C}^{d}$-valued standard Brownian motion with
quadratic variation $tI_{d}$.

Observe that%
\begin{equation}
\left[  X\right]  (t)=\Psi^{+}t+\int \nolimits_{\left[  0,t\right]  }%
\int_{\mathbb{C}^{d}\backslash \{0\}}xx^{\ast}J_{+}(\mathrm{d}s,\mathrm{d}%
x)=\Psi^{+}t+\int \nolimits_{\left[  0,t\right]  }\int_{\mathbb{H}_{d(1)}%
^{+}\backslash \{0\}}VJ_{L_{d}}(\mathrm{d}s,\mathrm{d}V) \label{Quad1decomp}%
\end{equation}
and%
\begin{align}
\left[  Y\right]  (t)  &  =\Psi^{-}t+\int \nolimits_{\left[  0,t\right]  }%
\int_{\mathbb{C}^{d}\backslash \{0\}}xx^{\ast}J^{-}(\mathrm{d}s,\mathrm{d}%
x)=\Psi^{-}t-\int \nolimits_{\left[  0,t\right]  }\int_{\mathbb{C}%
^{d}\backslash \{0\}}\left(  -xx^{\ast}\right)  J_{L_{d}}(\mathrm{d}%
s,\mathrm{d}x)\nonumber \\
&  =\Psi^{-}t-\int \nolimits_{\left[  0,t\right]  }\int_{\mathbb{H}_{d(1)}%
^{-}\backslash \{0\}}VJ_{L_{d}}(\mathrm{d}s,\mathrm{d}V), \label{Quad2decomp}%
\end{align}
where $\mathbb{H}_{d(1)}^{-}$ denotes the cone of negative (nonpositive)
definite matrices of rank one in $\mathbb{H}_{d}$. The first assertion follows
from (\ref{LIDBV}). Finally, since $J_{L_{d}}$ is a Poisson random measure and
$\mathbb{H}_{d(1)}^{+}\backslash \{0\}$ and $\mathbb{H}_{d(1)}^{-}%
\backslash \{0\}$ are disjoint sets, from the last expressions in
(\ref{Quad1decomp}) and (\ref{Quad2decomp}) we have that $\left[  X\right]  $
and $\left[  Y\right]  $ are independent processes, although $X$ and $Y$ are not.
\end{proof}

Next we consider the matrix L\'{e}vy processes associated to the BGCD matrix
ensembles $(M_{d})_{d\geq1}$. We have the following two consequences of the
former results.

\begin{corollary}
\label{corBG}Let $M_{d}=\left \{  M_{d}(t):t\geq0\right \}  $ be the matrix
L\'{e}vy process associated to the BGCD random matrix ensembles.

a) Let $\mu$ be the infinitely divisible distribution with triplet $\left(
0,\psi,\nu \right)  $ associated to $M_{d}$ such that
\[
\int_{\left \vert x\right \vert \leq1}\left(  1\wedge x\right)  \nu
(\mathrm{d}x)<\infty,\  \  \nu((-\infty,0])=0\  \text{ and\ }\mathcal{\psi}%
_{0}:=\mathcal{\psi}-\int_{x\leq1}x\nu(\mathrm{d}x)\geq0.
\]
Let us consider the L\'{e}vy-It\^{o} decomposition of $M_{d}(t)$ in
$\overline{\mathbb{H}}_{d}^{+}$
\[
M_{d}(t)=\mathcal{\psi}_{0}tdI_{d}+\int \nolimits_{\left[  0,t\right]  }%
\int_{\mathbb{H}_{d(1)}^{+}\backslash \{0\}}VJ_{M_{d}}(\mathrm{d}%
s,\mathrm{d}V).
\]
Then there exists a L\'{e}vy process $X=\left \{  X(t):t\geq0\right \}  $ in
$\mathbb{C}^{d}$ such that $M_{d}(t)=\left[  X\right]  (t)$, where
\[
X(t)=\left \vert \mathcal{\psi}_{0}\right \vert ^{1/2}B_{I}(t)+\int
\nolimits_{\left[  0,t\right]  }\int_{\mathbb{C}^{d}\cap \{ \left \vert
x\right \vert \leq1\}}x\widetilde{J}(\mathrm{d}s,\mathrm{d}x)+\int
\nolimits_{\left[  0,t\right]  }\int_{\mathbb{C}^{d}\cap \{ \left \vert
x\right \vert >1\}}xJ(\mathrm{d}s,\mathrm{d}x)\text{, }t\geq0\text{,}%
\]
$B_{I}$ is a $\mathbb{C}^{d}$-valued standard Brownian motion with quadratic
variation $tI_{d}$, and the Poisson random measure $J$ is given by
$J=J_{M_{d}}\circ \varphi_{+}^{-1}$.

b) If $M_{d}$ has bounded variation then there exist L\'{e}vy processes
$X=\left \{  X(t):t\geq0\right \}  $ and $Y=\left \{  Y(t):t\geq0\right \}  $ in
$\mathbb{C}^{d}$ such that $M_{d}(t)=\left[  X\right]  (t)-\left[  Y\right]
(t),$ where $\left \{  \left[  X\right]  (t):t\geq0\right \}  $ and $\left \{
\left[  Y\right]  (t):t\geq0\right \}  $ are independent.
\end{corollary}

\section{Covariation matrix processes approximation}

We now consider approximation of general BGCD ensembles by BGCD matrix
compound Poisson processes which are covariation of $\mathbb{C}^{d}$-valued
L\'{e}vy processes.

The following results gives realizations of BGCD ensembles of compound Poisson
type as the covariation of two $\mathbb{C}^{d}$-valued L\'{e}vy processes. Its
proof is straightforward.

\begin{proposition}
\label{BGCDCPP} Let $\mu$ be a compound Poisson distribution on $\mathbb{R}$
with L\'{e}vy measure $\nu$ and drift $\mathcal{\psi}\in \mathbb{R}$ and let
$(M_{d})_{d\geq1}$ be the BGCD matrix ensemble for $\Lambda(\mu).$ For each
$d\geq1$, assume that

i) $(\beta_{j})_{j\geq1}$ is a sequence of i.i.d. random variables with
distribution $\nu/\nu \left(  \mathbb{R}\right)  $.

ii) $(u_{j})_{j\geq1}$ is a sequence of i.i.d. random vectors with uniform
distribution on the unit sphere of $\mathbb{C}^{d}$.

iii) $\left \{  N(t)\right \}  _{t\geq0}$ is a Poisson process with parameter one.

Assume that $(\beta_{j})_{j\geq1}$, $(u_{j})_{j\geq1}$ and $\left \{
N(t)\right \}  _{t\geq0}$ are independent. Then

a) $M_{d}$ has the same distribution as $M_{d}(1)$ where
\begin{equation}
M_{d}(t)=\mathcal{\psi}tI_{d}+\sum_{j=1}^{N(t)}\beta_{j}u_{j}u_{j}^{\ast
},\quad t\geq0. \label{BGCP1}%
\end{equation}

b) $M_{d}(\cdot)=[X_{d},Y_{d}](\cdot)$ where $X_{d}=\left \{  X_{d}(t)\right \}
_{t\geq0},$ $Y_{d}=\left \{  Y_{d}(t)\right \}  _{t\geq0}$ are the
$\mathbb{C}^{d}$-valued L\'{e}vy processes%
\begin{equation}
X_{d}(t)=\sqrt{\left \vert \mathcal{\psi}\right \vert }B(t)+\sum_{j=1}%
^{N(t)}\sqrt{\left \vert \beta_{j}\right \vert }u_{j},\quad t\geq0,
\label{BGCP2}%
\end{equation}%
\begin{equation}
Y_{d}(t)=\mathrm{sign}\left(  \mathcal{\psi}\right)  \sqrt{\left \vert
\mathcal{\psi}\right \vert }B(t)+\sum_{j=1}^{N(t)}\mathrm{sign}\left(
\beta_{j}\right)  \sqrt{\left \vert \beta_{j}\right \vert }u_{j},\quad t\geq0,
\label{BGCP3}%
\end{equation}
and $B=\left \{  B(t)\right \}  _{t\geq0}$ is a $\mathbb{C}^{d}$-valued standard
Brownian motion independent of $(\beta_{j})_{j\geq1}$, $(u_{j})_{j\geq1}$ and
$\left \{  N(t)\right \}  _{t\geq0}$.
\end{proposition}

For the general case we have the following sample path approximation by
covariation processes for L\'{e}vy processes generated by the BGCD matrix ensembles.

\begin{theorem}
\label{General} Let $\mu$ be an infinitely divisible distribution on
$\mathbb{R}$ with triplet $(a^{2},\psi,\nu)$ and let $(M_{d})_{d\geq1}$ be the
corresponding BGCD matrix ensemble for $\Lambda(\mu).$ Let $d\geq1$ fixed and
assume that for $n\geq1$

i) $(\beta_{j}^{n})_{j\geq1}$ is a sequence of i.i.d. random variables with
distribution $\mu^{\ast \frac{1}{n}}$.

ii) $(u_{j}^{n})_{j\geq1}$ is a sequence of i.i.d. random vectors with uniform
distribution on the unit sphere of $\mathbb{C}^{d}$.

iii) $N^{n}=\left \{  N^{n}(t)\right \}  _{t\geq0}$ is a Poisson process with
parameter $n$.

iv) $B^{n}=\left \{  B^{n}(t)\right \}  _{t\geq0}$ is a $\mathbb{C}^{d}$-valued
standard Brownian motion.

v) $(\beta_{j}^{n})_{j\geq1}$, $(u_{j}^{n})_{j\geq1},N^{n}$ and $B^{n}$are independent.

Let
\begin{equation}
X_{d}^{n}(t)=\sqrt{\left \vert \mathcal{\psi}\right \vert }B^{n}(t)+\sum
_{j=1}^{N^{n}(t)}\sqrt{\left \vert \beta_{j}^{n}\right \vert }u_{j}^{n},\quad
t\geq0, \label{Gen1}%
\end{equation}%
\begin{equation}
Y_{d}^{n}(t)=\mathrm{sign}\left(  \mathcal{\psi}\right)  \sqrt{\left \vert
\mathcal{\psi}\right \vert }B^{n}(t)+\sum_{j=1}^{N^{n}(t)}\mathrm{sign}\left(
\beta_{j}^{n}\right)  \sqrt{\left \vert \beta_{j}^{n}\right \vert }u_{j}%
^{n},\quad t\geq0. \label{Gen2}%
\end{equation}
Then for each $d\geq1$ there exist $\mathbb{M}_{d}$-valued processes
$\widetilde{M}_{d}^{n}=\left \{  \widetilde{M}_{d}^{n}(t)\right \}  _{d\geq1}$
such that $\widetilde{M}_{d}^{n}\overset{\mathcal{L}}{=}[X_{d}^{n},Y_{d}^{n}%
]$,%
\[
\sup_{0<s\leq t}\left \Vert \widetilde{M}_{d}^{n}(s)-M_{d}(s)\right \Vert
\underset{n\rightarrow \infty}{\overset{\Pr}{\longrightarrow}}0\text{,\quad
}\forall t\geq0\text{,}%
\]
where $\left \{  M_{d}(t):t\geq0\right \}  $ is the $\mathbb{M}_{d}$-valued
L\'{e}vy process associated to $(M_{d})_{d\geq1}$.
\end{theorem}

\begin{proof}
By the compound Poisson approximation for infinitely divisible distributions
on $\mathbb{R}$ (see \cite[pp 45]{Sato1})$,$ we choose $\mu_{n}$ an infinitely
divisible distribution such that $\mu_{n}\longrightarrow \mu,$ where we take
the triplet of $\mu_{n}$ as $\left(  0,\psi^{n},\nu^{n}\right)  ,$ $\psi
^{n}=\int \frac{x}{1+\left \vert x\right \vert ^{2}}\nu^{n}\left(  dx\right)  $
and $\nu^{n}=n\mu^{\ast \frac{1}{n}}$, satisfying (see \cite[Theorem
8.7]{Sato1}) that for every bounded continuous function $f$ vanishing in a
neighborhood of zero%
\begin{equation}
\int_{\mathbb{R}}f\left(  r\right)  \nu^{n}\left(  dr\right)  \longrightarrow
\int_{\mathbb{R}}f\left(  r\right)  \nu \left(  dr\right)  \text{ as
}n\rightarrow \infty \text{,} \label{lev}%
\end{equation}
for each $\varepsilon>0$%
\begin{equation}
\int_{\left \vert r\right \vert \leq \varepsilon}r^{2}\nu^{n}\left(  dr\right)
\longrightarrow a^{2}\text{ as }n\rightarrow \infty, \label{gau}%
\end{equation}
and $\psi^{n}\rightarrow \psi$.

A similar proof as for Proposition \ref{BGCDCPP} gives%
\[
M_{d}^{n}(t):=\left[  X_{d}^{n},Y_{d}^{n\ast}\right]  (t)=\mathcal{\psi
}t\mathrm{I}_{d}+\sum_{j=0}^{N^{n}(t)}\beta_{j}^{n}u_{j}^{n}u_{j}^{n\ast},
\]
which is a matrix value compound Poisson process with triplet $\left(
\mathcal{A}_{d}^{n},\psi_{d}^{n},\nu_{d}^{n}\right)  $ given by $\mathcal{A}%
_{d}^{n}=0,\  \psi_{d}^{n}=\psi \mathrm{I}_{d}$ and%
\begin{equation}
\nu_{d}^{n}\left(  E\right)  =d\int_{\mathbb{S}(\mathbb{H}_{d(1)})}\int
_{0}^{\infty}1_{E}\left(  rV\right)  \nu_{V}^{n}\left(  \mathrm{d}r\right)
\Pi \left(  \mathrm{d}V\right)  ,\quad E\in \mathcal{B}\left(  \mathbb{H}%
_{d}\backslash \left \{  0\right \}  \right)  \text{,} \label{nudn}%
\end{equation}
where $\nu_{V}^{n}=\nu^{n}|_{(0,\infty)}$ or $\nu^{n}|_{(-\infty,0)}$
according to $V\geq0$ or\ $V\leq0$ and $\Pi$ is the measure on $\mathbb{S}%
(\mathbb{H}_{d(1)})$ in (\ref{pi}).

We will prove that $M_{d}^{n}\overset{\mathcal{L}}{\longrightarrow}M_{d}$ by
showing that the triplet $\left(  \mathcal{A}_{d}^{n},\psi_{d}^{n},\nu_{d}%
^{n}\right)  $ converges to the triplet $\left(  \mathcal{A}_{d},\psi_{d}%
,\nu_{d}\right)  $ of the BGCD matrix ensemble in Proposition \ref{polar} in
the sense of Proposition \ref{convternas}:

We observe that $\psi_{d}^{n}=\psi \mathrm{I}_{d}$ for each $n$.

Let $f:\mathbb{H}_{d(1)}\longrightarrow \mathbb{R}$ be a continuous bounded
function vanishing in a neighborhood of zero. Using the polar decomposition
(\ref{PDBGCD}) for $\nu_{d}^{n}$ we have%
\begin{align}
\int_{\mathbb{H}_{d(1)}}f\left(  \xi \right)  \nu_{d}^{n}\left(  d\xi \right)
&  =d\int_{\mathbb{S}(\mathbb{H}_{d(1)})}\int_{0}^{\infty}f\left(  rV\right)
\nu_{V}^{n}\left(  dr\right)  \Pi \left(  dV\right) \nonumber \\
&  =d\int_{\mathbb{S}(\mathbb{H}_{d(1)})\cap \overline{\mathbb{H}}_{d}^{+}}%
\int_{\left \{  -1,1\right \}  }\int_{0}^{\infty}f\left(  trV\right)  \nu
_{V}^{n}\left(  dr\right)  \lambda^{n}\left(  dt\right)  \omega_{d}\left(
dV\right)  . \label{intfb}%
\end{align}
For $V\in \mathbb{S}(\mathbb{H}_{d(1)})\cap \overline{\mathbb{H}}_{d}^{+}$
fixed,
\begin{align*}
\int_{\left \{  -1,1\right \}  }\int_{0}^{\infty}f\left(  trV\right)  \nu
_{V}^{n}\left(  dr\right)  \lambda^{n}\left(  dt\right)   &  =\lambda
^{n}\left(  \left \{  1\right \}  \right)  \int_{0}^{\infty}f\left(  rV\right)
\nu^{n}\left(  dr\right) \\
&  +\lambda^{n}\left(  \left \{  -1\right \}  \right)  \int_{-\infty}%
^{0}f\left(  rV\right)  \nu^{n}\left(  dr\right)  \text{.}%
\end{align*}
As a function of $r$, $f\left(  rV\right)  $ is a real valued continuous
bounded function vanishing in a neighborhood of zero, hence using (\ref{lev})%
\[
\lambda^{n}\left(  \left \{  1\right \}  \right)  \int_{0}^{\infty}f\left(
rV\right)  \nu^{n}\left(  dr\right)  \longrightarrow \lambda \left(  \left \{
1\right \}  \right)  \int_{0}^{\infty}f\left(  rV\right)  \nu \left(  dr\right)
\]
and%
\[
\lambda^{n}\left(  \left \{  -1\right \}  \right)  \int_{-\infty}^{0}f\left(
rV\right)  \nu^{n}\left(  dr\right)  \longrightarrow \lambda \left(  \left \{
-1\right \}  \right)  \int_{-\infty}^{0}f\left(  rV\right)  \nu \left(
dr\right)  .
\]
Then from (\ref{intfb})%
\begin{align*}
\int_{\mathbb{H}_{d(1)}}f\left(  \xi \right)  \nu_{d}^{n}\left(  d\xi \right)
&  \longrightarrow d\int_{\mathbb{S}(\mathbb{H}_{d(1)})\cap \overline
{\mathbb{H}}_{d}^{+}}\int_{\left \{  -1,1\right \}  }\int_{0}^{\infty}f\left(
trV\right)  \nu_{V}\left(  dr\right)  \lambda \left(  dt\right)  \omega
_{d}\left(  dV\right) \\
&  =d\int_{\mathbb{S}(\mathbb{H}_{d(1)})}\int_{0}^{\infty}f\left(  rV\right)
\nu_{d}\left(  dr\right)  \Pi \left(  dV\right)  =\int_{\mathbb{H}_{d(1)}%
}f\left(  \xi \right)  \nu_{d}\left(  d\xi \right)  .
\end{align*}

Next, we verify the convergence of the Gaussian part.

Let us define, for each $\varepsilon>0$ and $n\geq1,$ the operator
$\mathcal{A}^{n,\varepsilon}:\mathbb{H}_{d}\longrightarrow \mathbb{H}_{d}$ by
\[
\mathrm{tr}\left(  \Theta \mathcal{A}^{n,\varepsilon}\Theta \right)
=\int_{\left \Vert \xi \right \Vert \leq \varepsilon}\left \vert \mathrm{tr}\left(
\Theta \xi \right)  \right \vert ^{2}\nu_{d}^{n}\left(  d\xi \right)  .
\]
From (\ref{nudn}) we get
\begin{align*}
&  \int_{\left \Vert \xi \right \Vert \leq \varepsilon}\left \vert \mathrm{tr}%
\left(  \Theta \xi \right)  \right \vert ^{2}\nu_{d}^{n}\left(  d\xi \right)
=d\int_{\mathbb{S}(\mathbb{H}_{d(1)})}\int_{0}^{\infty}\mathbf{1}_{\left \{
\left \Vert rV\right \Vert \leq \varepsilon \right \}  }\left(  rV\right)
\left \vert \mathrm{tr}\left(  r\Theta V\right)  \right \vert ^{2}\nu_{V}%
^{n}\left(  dr\right)  \Pi \left(  dV\right) \\
&  =d\int_{\mathbb{S}(\mathbb{H}_{d(1)})\cap \overline{\mathbb{H}}_{d}^{+}}%
\int_{\left \{  -1,1\right \}  }\int_{0}^{\infty}\mathbf{1}_{\left \{
r\leq \varepsilon \right \}  }\left(  rtV\right)  r^{2}\left \vert \mathrm{tr}%
\left(  \Theta V\right)  \right \vert ^{2}\nu_{V}^{n}\left(  dr\right)
\lambda \left(  dt\right)  \omega_{d}\left(  dV\right) \\
&  =d\int_{\mathbb{S}(\mathbb{H}_{d(1)})\cap \overline{\mathbb{H}}_{d}^{+}}%
\int_{\mathbb{R}}\mathbf{1}_{\left \{  r\leq \varepsilon \right \}  }\left(
rV\right)  r^{2}\left \vert \mathrm{tr}\left(  \Theta V\right)  \right \vert
^{2}\nu^{n}\left(  dr\right)  \omega_{d}\left(  dV\right) \\
&  =d\int_{\mathbb{S}(\mathbb{H}_{d(1)})\cap \overline{\mathbb{H}}_{d}^{+}%
}\left \vert \mathrm{tr}\left(  \Theta V\right)  \right \vert ^{2}%
\int_{\left \vert r\right \vert \leq \varepsilon}r^{2}\nu^{n}\left(  dr\right)
\omega_{d}\left(  dV\right)  .
\end{align*}
Then using (\ref{gau}),%
\[
\int_{\left \Vert \xi \right \Vert \leq \varepsilon}\left \vert \mathrm{tr}\left(
\Theta \xi \right)  \right \vert ^{2}\nu_{d}^{n}\left(  d\xi \right)
\longrightarrow da^{2}E_{u}\left \vert \mathrm{tr}\left(  \Theta uu^{\ast
}\right)  \right \vert ^{2}\text{,}%
\]
where $u$ is a uniformly distributed column random vector in the unit sphere
of $\mathbb{C}^{d}$. Finally
\begin{equation}
da^{2}E_{u}\left \vert \mathrm{tr}\left(  \Theta uu^{\ast}\right)  \right \vert
^{2}=\frac{a^{2}}{d+1}\left(  \mathrm{tr}\left(  \Theta^{2}\right)  +\left(
\mathrm{tr}\left(  \Theta \right)  \right)  ^{2}\right)  =\mathrm{tr}\left(
\Theta^{\ast}\mathcal{A}_{d}\Theta^{\ast}\right)  , \label{covgau}%
\end{equation}
\newline where $\mathcal{A}_{d}$ is as in (\ref{GPBGCD}) and the first
equality in (\ref{covgau}) follows from page $637$ in \cite{CD}. Thus
$M_{d}^{n}\overset{\mathcal{L}}{\longrightarrow}M_{d}$ and the conclusion
follows from Proposition \ref{convproc}.
\end{proof}

\section{Final remarks}

\begin{enumerate}
\item For the present work we do not have a specific financial application in
mind. However, infinitely divisible nonnegative definite matrix processes with
rank one jumps as characterized in Theorem \ref{Sub}, might be useful in the
study of multivariate high-frequency data using realized covariation, where
matrix covariation processes appear; see for example \cite{BNSh04}. Moreover,
it seems interesting to explore the construction of financial oriented matrix
L\'{e}vy based models as in \cite{BNSe09} for the specific case of rank one
jumps matrix process of bounded variation.

\item In the direction of free probability, it is well known that the
so-called Hermitian Brownian motion matrix ensemble $\left \{  B_{d}%
(t):t\geq0\right \}  $, $d\geq1,$ is a realization of the free Brownian motion.
It is an open question if the matrix L\'{e}vy processes from BGCD models
$\left \{  M_{d}(t):t\geq0\right \}  $, $d\geq1$, are realizations of free
L\'{e}vy processes. A first step in this direction would be to prove that the
increments of a BGCD ensemble become free independent. A second step, more
related to our work, would be to have an insight of the implication of the
rank one condition of the matrix L\'{e}vy BGCD process in Corollary
\ref{corBG} as realization of a positive free L\'{e}vy process. These two
problems are the subjects of current research of one of the coauthors.

\item In \cite{BG07} a new Bercovici-Pata bijection for certain free
convolution $\boxplus_{c}$ is established and a $d\times d^{\prime}$ random
matrix model for this bijection which is very close to the one given by the
BGCD random matrix model is established. It can be seen that the L\'{e}vy
measures of these rectangular BGCD random matrices are supported in the subset
of $d\times d^{\prime}$ complex matrices of rank one, in a similar way as done
in \cite{DRA} for the BGCD case. It would be of interest to have the analogue
results on bounded variation of Section 4 for the L\'{e}vy processes
associated to these rectangular BGCD random matrices, considering an
appropriate nonnegative definite notion for rectangular matrices.
\end{enumerate}

\noindent \textbf{Acknowledgement}. \emph{This work was done while Victor
P\'{e}rez-Abreu was visiting Universidad Aut\'{o}noma de Sinaloa in January
and May of 2012. The authors thank two referees for the very carefully and
detailed reading of a previous version of the manuscript and for their
comments that improved Theorems \ref{Sub} and \ref{Bdv} and the presentation
of the present version of the manuscript.}

\end{document}